\documentclass[a4paper,12pt]{amsart}
\usepackage{amsmath,amsthm,amssymb}
\usepackage{mathrsfs}
\usepackage{mathtools}
\usepackage{ifthen}
\usepackage{booktabs}
\usepackage{caption}
\usepackage{placeins}
\usepackage{makecell}
\usepackage{float}
\usepackage{cite}
\usepackage{mathrsfs}
\nonstopmode \numberwithin{equation}{section}
\newtheorem{thm}{Theorem}%[section]

\newtheorem{cor}{Corollary}%[section]
\newtheorem{lem}{Lemma}%[section]
\newtheorem{conj}{Conjecture}

\theoremstyle{definition}
\newtheorem{defn}{Definition}[section]

\newtheorem{prob}[equation]{Problem}

\newenvironment{rem}{%
\bigskip
\noindent \textsl{{\sl Remark. }}}{\bigskip}
\newenvironment{rems}{%
\bigskip
\noindent \textsl{{\sl Remarks. }}}{\bigskip}

\newcounter {own}
\def\theown {\thesection       .\arabic{own}}

\newenvironment{pf}[1][]{%
 \vskip 3mm
 \noindent
 \ifthenelse{\equal{#1}{}}%
  {{\slshape Proof. }}%
  {{\slshape #1.} }%
 }%
{\qed\bigskip}

\newcounter{alphabet}

\newcommand{\A}{{\mathcal A}}

\newcommand{\es}{{\mathcal S}}

\newcommand{\ID}{{\mathbb D}}

\newcommand{\IC}{{\mathbb C}}

\newcommand{\D}{{\mathbb D}}
\def\be{\begin{equation}}
\def\ee{\end{equation}}

\newcommand{\bee}{\begin{enumerate}}
\newcommand{\eee}{\end{enumerate}}

\newcommand{\blem}{\begin{lem}}
\newcommand{\elem}{\end{lem}}
\newcommand{\bthm}{\begin{thm}}
\newcommand{\ethm}{\end{thm}}
\newcommand{\bcor}{\begin{cor}}
\newcommand{\ecor}{\end{cor}}
\newcommand{\beg}{\begin{examp}}
\newcommand{\eeg}{\end{examp}}
\newcommand{\begs}{\begin{examples}}
\newcommand{\eegs}{\end{examples}}
\newcommand{\bdefe}{\begin{defn}}
\newcommand{\edefe}{\end{defn}}
\newcommand{\bprob}{\begin{prob}}
\newcommand{\eprob}{\end{prob}}
\newcommand{\bei}{\begin{itemize}}
\newcommand{\eei}{\end{itemize}}

\newcommand{\bcon}{\begin{conj}}
\newcommand{\econ}{\end{conj}}
\newcommand{\bcons}{\begin{conjs}}
\newcommand{\econs}{\end{conjs}}
\newcommand{\bprop}{\begin{propo}}
\newcommand{\eprop}{\end{propo}}
\newcommand{\br}{\begin{rem}}
\newcommand{\er}{\end{rem}}
\newcommand{\brs}{\begin{rems}}
\newcommand{\ers}{\end{rems}}
\newcommand{\bo}{\begin{obser}}
\newcommand{\eo}{\end{obser}}
\newcommand{\bos}{\begin{obsers}}
\newcommand{\eos}{\end{obsers}}
\newcommand{\bpf}{\begin{pf}}
\newcommand{\epf}{\end{pf}}
\newcommand{\ba}{\begin{array}}
\newcommand{\ea}{\end{array}}
\newcommand{\beq}{\begin{eqnarray}}
\newcommand{\beqq}{\begin{eqnarray*}}
\newcommand{\eeq}{\end{eqnarray}}
\newcommand{\eeqq}{\end{eqnarray*}}

\newcounter{minutes}
\divide\time by 60
\newcounter{hours}
\multiply\time by 60 \addtocounter{minutes}{-\time}
\begin{document}
\title{Radius of starlikeness of $\mathcal{S}\ast \mathcal {S}t(\alpha)$ }
\begin{center}
{\tiny \texttt{FILE:~\jobname .tex,
        printed: \number\year-\number\month-\number\day,
        \thehours.\ifnum\theminutes<10{0}\fi\theminutes}
}
\end{center}
\author{Bappaditya Bhowmik${}^{~\mathbf{*}}$}
\address{Bappaditya Bhowmik, Department of Mathematics,
Indian Institute of Technology Kharagpur, Kharagpur - 721302, India.}
\email{bappaditya@maths.iitkgp.ac.in}
\author{Souvik Biswas}
\address{Souvik Biswas, Department of Mathematics, Indian Institute of Technology Kharagpur, Kharagpur - 721302, India.}
\email{souvikbiswas158@gmail.com}


\begin{abstract}
   Let $\mathcal{S}$ be the set of all analytic univalent functions $f$ defined in the open unit disc $\D$, with $f(0)=0=f'(0)-1$. For $\alpha\in[0,1)$, let $\mathcal{S}t(\alpha)$ be the set of all starlike functions of order $\alpha$ in $\mathcal{S}$. In this article, by applying duality technique we obtain the radius of a disc that is mapped onto a starlike domain with respect to the origin by the functions in the set $\mathcal{S}\ast \mathcal {S}t(\alpha):=\{f\ast g :f\in\mathcal{S},~g\in\mathcal{S}t(\alpha)\}$. Here, `$\ast$' denotes the convolution (or Hadamard product) of two analytic functions in $\mathbb{D}$.
\end{abstract}

\maketitle
\pagestyle{myheadings}
\markboth{B. Bhowmik and S. Biswas}{}
\maketitle
\pagestyle{myheadings}
\markboth{B. Bhowmik and S. Biswas }{Radius of starlikeness of $\mathcal{S}\ast \mathcal {S}t(\alpha)$}

\bigskip
\bigskip

\section{Introduction and statement of the main result}\label{P5sec1}
 Let $\mathbb{C}$ be the whole complex plane and $\mathbb{D}:=\{z\in\mathbb{C}:|z|<1\}$ be the open unit disc. We denote the unit circle by $\partial \D:=\{z\in \mathbb{C}:|z|=1\}$. Let $\mathcal{A}$ be the set of all analytic functions $f$ defined in $\mathbb{D}$ with the normalization $f(0)=0=f'(0)-1$. Let $\mathcal{S}$ denote the set of all univalent functions in $\mathcal{A}$. Over the years, various subsets of $\mathcal{S}$ characterized by specific geometric properties have been investigated. In \cite{rob}, Robertson introduced the sets $\mathcal{C}(\alpha)$ and $\mathcal{S}t(\alpha)$ of convex and starlike functions of order $\alpha<1$, respectively, defined as follows.
\beqq
\mathcal{C}(\alpha)&:=&\{f\in \mathcal{A} : {\rm{Re}}\,\left(1+\frac{zf''(z)}{f'(z)}\right)>\alpha, ~ z\in\D\},\\
\mathcal{S}t(\alpha)&:=&\{f\in \mathcal{A} : {\rm{Re}}\,\left(\frac{zf'(z)}{f(z)}\right)>\alpha, ~ z\in \D\}.
\eeqq
It is well-known that every function in $\mathcal{C}(\alpha)$ and $\mathcal{S}t(\alpha)$ is univalent for $\alpha\in[0,1)$. However, if $\alpha<0$ then functions in these sets need not be univalent. We denote the set of all convex functions by $\mathcal{C}:=\mathcal{C}(0)$, which consists of all functions $f\in\mathcal{S}$ such that $f$ maps $\D$ conformally onto a convex domain. Similarly, the set of starlike functions is denoted by $\mathcal{S}t:=\mathcal{S}t(0)$, which consists of all functions $f \in \mathcal{S}$ such that $f$ maps $\D$ conformally onto a starlike domain with respect to the origin. In $1933$, E. Strohh\"acker (see \cite{str}) proved that $\mathcal{C}\subsetneq \mathcal{S}t(1/2)$ and the constant $1/2$ cannot be improved. In other words, $\mathcal{C}\subsetneq \mathcal{S}t(\alpha)$ for $\alpha\in[0,1/2]$. The radius of convexity (or starlikeness) of a subset $\A_1$ of $\A$ is the largest number $r\in (0,1]$ such that each function $f\in \A_1$ is convex (or starlike) in $\ID_r=\{z\in \IC: |z|<r\}$. In $1920$, Nevanlinna (see \cite{nev}) proved that the radius of convexity of $\es$ is $2-\sqrt{3}$. In $1934$, Grunsky (see\cite[p.~141]{grunsky}) obtained the radius of starlikeness of $\mathcal{S}$ as $\tanh{\pi/4}$. We also refer the articles \cite{ba, kumar, szak, ob, sokol} for various studies on radii of convexity and starlikeness for analytic functions. We present here the definition of convolution (or Hadamard product) of two analytic functions which will be required to describe the main aim of this article. The convolution (or Hadamard product) of two functions $f,g\in \mathcal{A}$ with the power series expansions $f=\sum_{n=0}^\infty a_nz^n$ and $g(z)=\sum_{n=0}^\infty b_nz^n$ for $z\in\mathbb{D}$, is defined as 
$$
 (f\ast g)(z)=\sum_{n=0}^\infty a_n b_n z^n, \quad z\in\D.
$$
It can be easily verified that $f\ast g \in \mathcal{A}$. In \cite{sheil}, Ruscheweyh and Sheil-Small proved that $\mathcal{C}\ast \mathcal{C} \subseteq \mathcal{C}$ and $\mathcal{C}\ast \mathcal{K} \subseteq \mathcal{K}$, where $\mathcal{K}$ denotes the set of close-to-convex functions in $\mathcal{S}$. We recall that a function $f\in\A$ with $f(0)=0=f'(0)-1$ is said to be a close-to-convex function if there is a convex function $h$ such that
$$
  {\rm{Re}}\,\left(\frac{f'(z)}{h'(z)}\right)>0, \quad z\in\D.
$$
For $r\in\left(\tanh{\left(\pi/4\right)},1\right)$, let $a=a(r)=(1+r^2)(1-r^2)^{-1}$, and let $x_0(r)$ be the unique root of the polynomial
$$
 x^3-ax^2+a^2x-a=0.
$$
The radius of close-to-convexity of $\mathcal{S}$ is the unique real root $r_{cc}~(=0.8098\cdots)$ of the equation 
\begin{equation*}
    2 \operatorname{arccot}{\left(\frac{1-r^2}{1+r^2}x_0(r)\right)}+\log{\left(1+x_0^2(r)\right)}-2\log{\left(\frac{2r}{1-r^2}\right)}=0,
\end{equation*}
contained in the interval $(\tanh{(\pi/4)}, 1)$ (see \cite[Theorem~$3.2.5$]{graham}). In \cite{sheil}, Ruscheweyh and Sheil-Small proved that the set $\mathcal{S}t (1/2)$ is closed under convolution, i.e. $\mathcal{S}t(1/2)\ast\mathcal{S}t(1/2)\subseteq\mathcal{S}t(1/2)$. In \cite{small}, it is shown that for $f\in \mathcal{S}t$ and $g\in \mathcal{S}t$, $f\ast g$ need not be in $\mathcal{S}$. Later in $1997$, Y. Ling and S. Ding (\cite{ling}) obtained the radius of starlikeness and convexity of the set $\mathcal{S}t \ast \mathcal{S}t$ as $2-\sqrt{3}$ and $5-2\sqrt{6}$, respectively. However, so far only a few geometric properties of the sets $\mathcal{S} \ast \mathcal{S}$ and $\mathcal{S} \ast \mathcal{K}$ are known. Using Ruscheweyh's well-known duality principle (c.f.~\cite{ru}), in $2003$ Richard Greiner and Oliver Roth proved that the radius of convexity of the set $\mathcal{S} \ast \mathcal{K}$ is $5-2\sqrt{6}$ (see \cite[Theorem~$2.1$]{roth}) and conjectured that the radius of convexity of the set $\mathcal{S} \ast \mathcal{S}$ is $5-2\sqrt{6}$ (see \cite[Conjecture~$2.3$]{roth}), which is still open.

Our main concern is to obtain the radius of starlikeness of the set $\mathcal{S} \ast \mathcal{S}t(\alpha)$, $\alpha\in[0,1)$. For $\alpha=0$, the radius of starlikeness of the set $\mathcal{S}\ast\mathcal{S}t(0)=\mathcal{S}\ast\mathcal{S}t$ is $2-\sqrt{3}$ and the function $k\ast k=zk'(z)$ shows that $2-\sqrt{3}$ cannot be replaced by any larger number, where $k(z)=z/(1-z)^2$. This follows directly from the inclusion $\mathcal{C} \ast \mathcal{S}t\subseteq \mathcal{S}t$ (see \cite{sheil}) together with the fact that the radius of convexity of the set $\mathcal{S}$ is $2-\sqrt{3}$. Let $\alpha_1~(=0.3349\cdots)$ be the smallest positive root of the equation $20\alpha^4-52\alpha^3+15\alpha^2+12\alpha-4=0$. Then for $\alpha\in(0,1)$, it can be readily seen that a lower bound for the radius of starlikeness of the set $\mathcal{S} \ast \mathcal{S}t(\alpha)$ is $r_c(\alpha)\tanh{(\pi/4)}$, where
$$
  r_c(\alpha):=
        \begin{cases}
            \frac{1}{2-3\alpha+\sqrt{5\alpha^2-8\alpha+3}}, & \text{if} ~~\alpha\in\left(0,\alpha_1\right]\\[4pt]
            \left(\frac{5\alpha-1}{4\alpha^2-\alpha+1+4\alpha\sqrt{\alpha^2-3\alpha+2}}\right)^{1/2}, & \text{if} ~~\alpha\in\left(\alpha_1,1\right).
        \end{cases}
$$
The above result is a direct consequence of the inclusion $\mathcal{S}t\ast \mathcal{C} \subseteq \mathcal{S}t$, since the radius of starlikeness of $\mathcal{S}$ is $\tanh{(\pi/4)}$ and the radius of convexity of $\mathcal{S}t(\alpha)$ is $r_c(\alpha)$ (see \cite{schild, goel}). For $\alpha\in(0,1)$, we obtain another lower bound for the radius of starlikeness of the set $\mathcal{S} \ast \mathcal{S}t(\alpha)$ as $2-\sqrt{3}$, which can be established from the inclusion $\mathcal{C} \ast \mathcal{S}t\subseteq \mathcal{S}t$, because, the radius of convexity of the set $\mathcal{S}$ is $2-\sqrt{3}$ and $\mathcal{S}t(\alpha)\subsetneq\mathcal{S}t$, $\alpha\in(0,1)$. Thus, each function in $\mathcal{S} \ast \mathcal{S}t(\alpha)$, $\alpha\in[0,1)$, is starlike in $|z|<R_0(\alpha):=\max{\left\{2-\sqrt{3},~r_c(\alpha)\tanh{(\pi/4)}\right\}}$. Let  $\alpha_2:=\left(2-3b+\sqrt{5b^2+4b}\right)/4=0.2404\cdots$, where $b=\left(2+\sqrt{3}\right)\tanh{(\pi/4)}$. Then a little computation yields
\begin{equation}\label{P5eq1.1}
    R_0(\alpha)=
    \begin{cases}
        2-\sqrt{3}=0.2679\cdots, & \text{if} ~~\alpha\in\left[0,\alpha_2\right]\\[4pt]
        \frac{\tanh{(\pi/4)}}{2-3\alpha+\sqrt{5\alpha^2-8\alpha+3}}, & \text{if} ~~\alpha\in\left(\alpha_2,\alpha_1\right]\\[4pt]
        \left(\frac{5\alpha-1}{4\alpha^2-\alpha+1+4\alpha\sqrt{\alpha^2-3\alpha+2}}\right)^{1/2}\tanh{(\pi/4)}, & \text{if} ~~\alpha\in\left(\alpha_1,1\right).
    \end{cases}
\end{equation}
In this article, using duality between certain subsets of $\mathcal{S}$, we obtain improved radius of starlikeness of the set $\mathcal{S} \ast \mathcal{S}t(\alpha)$, for $\alpha\in(\alpha_0,1)$, where $\alpha_0=0.1361\cdots$. We now briefly describe it here. Let $\mathcal{M}$ be the set of all functions $f\in\mathcal{A}$ such that $f\ast g \in \mathcal{S}$ for every $g\in \mathcal{C}$. Clearly, $\mathcal{M}\subseteq\mathcal{S}$, which can be seen by taking $g(z)=z/(1-z)\in\mathcal{C}$. Let $f\in \mathcal{K}$. Then we have $f\ast g \in \mathcal{K}\ast \mathcal{C}$, for all $g\in\mathcal{C}$. Since $\mathcal{C}\ast \mathcal{K}\subseteq \mathcal{K}$, it follows that $f\ast g \in \mathcal{K}\subset \mathcal{S}$. This implies $f\in \mathcal{M}$, and hence $\mathcal{K}\subseteq \mathcal{M}$. On the other hand, from the definition of $\mathcal{M}$ itself, it follows that if $f\in\mathcal{M}$, then $(f\ast g)'(z)\ne 0$, $z\in\D$ for all $g\in\mathcal{C}$. Thus, we get $f(z)\ast (zg'(z))\ne 0$, $0<|z|<1$. By virtue of the Alexander's theorem (see \cite{alex}), $g\in \mathcal{C}$ if and only if $zg'(z) \in \mathcal{S}t$. It follows that $f(z)\ast h(z)\ne 0$, $0<|z|<1$, where $h(z)=zg'(z)\in \mathcal{S}t$. Therefore, if $f\in \mathcal{M}$, then $(f\ast g)(z)\ne 0$ for all $g\in \mathcal{S}t$. In other words, the set $\mathcal{M}$ is dual to $\mathcal{S}t$. Using this duality between the sets $\mathcal{M}$ and $\mathcal{S}t$, and some specific convolution techniques, we obtain the following result.
\begin{thm}\label{P5thm1}
    Let $\alpha\in[0,1)$. Let $\alpha_0~(=0.1361\cdots)$ be the smallest positive root of the equation 
    $$
      \zeta_\alpha\left(\frac{2-\sqrt{3}}{r_{cc}}\right)=0,
    $$
    where $r_{cc}$ denotes the radius of close-to-convexity of the set $\mathcal{S}$ and 
    \begin{multline*}
        \zeta_\alpha(r):=(16\alpha^3-8\alpha^2+\alpha-1)r^4+4(1-\alpha)(1-2\alpha)(3-4\alpha)r^3\\+2(8\alpha^3-44\alpha^2+47\alpha-19)r^2+12(1-\alpha)(1-2\alpha)r+9\alpha-1.
    \end{multline*}
    Then the radius of starlikeness of the set $\mathcal{S}\ast\mathcal{S}t(\alpha)$ is at least
    $$
      R(\alpha):=\begin{cases}
                2-\sqrt{3}, & \text{if} ~~\alpha\in[0,\alpha_0]\\
                R_1(\alpha)r_{cc}, & \text{if} ~~\alpha\in(\alpha_0,1),
                \end{cases}
    $$
    where $R_1(\alpha)$ is the least value of $r\in\left(3-2\sqrt{2},1\right)$ that satisfies the equation $\zeta_\alpha(r)=0$.
\end{thm}
As mentioned before, in Theorem~\ref{P5thm1}, we obtain an improved lower bound $R(\alpha)$ for the radius of starlikeness of the set $\mathcal{S} \ast \mathcal{S}t(\alpha)$, $\alpha\in(\alpha_0,1)$. In order to show this improvement, in the following table (Table~\ref{P5tab1}), we list the values of $R_0(\alpha)$ (from (\ref{P5eq1.1})) and $R(\alpha)$ (from Theorem~{\ref{P5thm1}}) for different values of $\alpha\in[0.14,1)$.
\captionsetup[table]{skip=8pt, width=\textwidth, justification=centering}
    \begin{table}[H]
    \centering
    \renewcommand{\arraystretch}{1.2}
 	\begin{tabular}{{c@{\hspace{1cm}}c@{\hspace{1cm}}c}}
    \toprule
 		  \textbf{Values of $\alpha$} & \makecell{\textbf{Values of $R_0(\alpha)$}\\ \textbf{from \eqref{P5eq1.1}}} & \makecell{\textbf{Values of $R(\alpha)$}\\\textbf{from Theorem~\ref{P5thm1}}}\\
        \midrule
              $0.14$ & $0.2679\cdots$ & $0.2701\cdots$\\
              $0.3$ & $0.3086\cdots$ & $0.3538\cdots$\\
 			 $0.5$ & $0.4467\cdots$ & $0.4535\cdots$\\
              $0.7$ & $0.5178\cdots$ & $0.5516\cdots$\\
              $0.9$ & $0.5761\cdots$ & $0.6614\cdots$\\ 
        \bottomrule
 	\end{tabular}
    \caption{Values of $R_0(\alpha)$ and $R(\alpha)$ for $\alpha\in[0.14,1)$.}
    \label{P5tab1}
 	\end{table}

\section{Proof of the main result}
To prove Theorem~\ref{P5thm1}, we need to prove several lemmas (Lemma~\ref{P5lem1}--\ref{P5lem5}). At first we prove them below.
\begin{lem}\label{P5lem1}
    For each fixed $\alpha\in(0,1)$, we have
    $$
      G_\alpha(r,\theta)>0, \quad r\in[0,1),~ \theta\in[-\pi,\pi],
    $$
    where
     \begin{align}\label{P5eq2.1}
        G_\alpha(r,\theta)& := (4\alpha^2-4\alpha+1)r^5+\alpha r^4+(8\alpha^2-3\alpha-1)r^3+(5\alpha-1)r^2\\& \quad +\alpha r+1-r\big((12\alpha^2-9\alpha+1)r^3+(5\alpha-1)r^2+(4\alpha^2+\alpha-1)r\notag\\& \quad +3\alpha+1\big)\cos{\theta}+2\alpha r^2\big((2\alpha-1)r+1\big)\cos{(2\theta)}.\notag
    \end{align}
\end{lem}
\begin{pf}
   Let $\alpha\in(0,1)$. A straightforward calculation shows that
   \begin{equation}\label{P5eq2.2}
       \frac{\partial}{\partial\theta} G_\alpha(r,\theta)=rq_\alpha(r,\theta)\sin{\theta},\quad r\in[0,1],~\theta\in[-\pi,\pi],
   \end{equation}
   where
   \begin{align*}
       q_\alpha(r,\theta)&=3\alpha+1+(4\alpha^2+\alpha-1)r+(5\alpha-1)r^2+(12\alpha^2-9\alpha+1)r^3\\& \quad -8\alpha r\left(1+(2\alpha-1)r\right)\cos{\theta}.
   \end{align*}
   Since $1+(2\alpha-1)r>0$ for all $\alpha\in(0,1)$ and $r\in[0,1]$, we have
   \begin{align*}
       q_\alpha(r,\theta)& \geq q_\alpha(r,0)=(1-r)\eta_\alpha(r), \quad r\in[0,1],~\theta\in[-\pi,\pi],
   \end{align*}
   where 
   $$
    \eta_\alpha(r)=3\alpha+1-4\alpha(1-\alpha)r-(12\alpha^2-9\alpha+1)r^2.
   $$
   Then for each fixed $\alpha\in(0,1)$, we have
   $$
   \eta_\alpha'(r)=-4\alpha(1-\alpha)-2(12\alpha^2-9\alpha+1)r, \quad r\in[0,1].
   $$
   Now we consider two cases to prove that for each $\alpha\in(0,1)$, $\eta_\alpha(r)>0$, $r\in[0,1]$.
   
   \medskip
   \noindent Case (i). Let $\alpha\in\left(0,2/\left(9+\sqrt{33}\right)\right]\cup \left[2/\left(9-\sqrt{33}\right),1\right)$. Then we have
   $$
     \eta_\alpha'(r)<0, \quad r\in[0,1].
   $$
   This shows $\eta_\alpha$ is a decreasing function of $r$, $r\in[0,1]$. Therefore, 
   $$
   \eta_\alpha(r)\geq\eta_\alpha(1)=8\alpha(1-\alpha)>0, \quad r\in[0,1].
   $$
   
   \medskip
   \noindent Case (ii). Let $\alpha\in\left(2/\left(9+\sqrt{33}\right),2/\left(9-\sqrt{33}\right)\right)$. For these values of $\alpha$, we have
   $$
     \eta_\alpha'(r)=-2(12\alpha^2-9\alpha+1)(r-r_0),\quad r\in[0,1],
   $$
   where
   $$
     r_0=\frac{2\alpha(\alpha-1)}{12\alpha^2-9\alpha+1}.
   $$
   Case (ii)(a). If $\alpha\in\left(2/\left(9+\sqrt{33}\right),1/5\right)$ then $12\alpha^2-9\alpha+1<0$ and $r_0>1$. This implies $\eta_\alpha'(r)<0$, $r\in[0,1]$. Therefore,
   $$
     \eta_\alpha(r)\geq \eta_\alpha(1)=8\alpha(1-\alpha)>0,\quad r\in[0,1].
   $$
   Case (ii)(b). If $\alpha\in[1/5,1/2]$ then $12\alpha^2-9\alpha+1<0$ and $r_0\in[0,1]$. Since
   $$
     \eta_\alpha''(r_0)=-2(12\alpha^2-9\alpha+1)>0,
   $$
   it follows that $\eta_\alpha(r)$ attains its minimum at $r=r_0$. Therefore,
   $$
     \eta_\alpha(r)\geq \eta_\alpha(r_0)=\frac{4\alpha^4+28\alpha^3-11\alpha^2-6\alpha+1}{12\alpha^2-9\alpha+1}>0, \quad r\in[0,1].
   $$
   Case (ii)(c). If $\alpha\in\left(1/2,2/\left(9-\sqrt{33}\right)\right)$ then $12\alpha^2-9\alpha+1<0$ and $r_0>1$. This implies $\eta_\alpha'(r)<0$, $r\in[0,1]$. Therefore,
   $$
     \eta_\alpha(r)\geq \eta_\alpha(1)=8\alpha(1-\alpha)>0,\quad r\in[0,1].
   $$
   
   \medskip
   \noindent Thus, combining all the above cases, we get $\eta_\alpha(r)>0$, $r\in[0,1]$ for all $\alpha\in(0,1)$. It follows that for each fixed $\alpha\in(0,1)$, we have $q_\alpha(r,\theta)>0$ for all $r\in[0,1)$, $\theta\in[-\pi,\pi]$. Hence, for each fixed $\alpha\in(0,1)$, from \eqref{P5eq2.2} we have
   $$
    \frac{\partial G_\alpha}{\partial\theta}<0 \quad {\rm{if}}~~-\pi<\theta <0, \qquad {\rm{and}} \qquad \frac{\partial G_\alpha}{\partial\theta}>0 \quad {\rm{if}}~~0<\theta<\pi,
   $$
   for all $r\in(0,1)$. Therefore,
   \begin{equation}\label{P5eq2.3}
       \min_{-\pi\leq \theta\leq\pi}{G_\alpha(r,\theta)}=G_\alpha(r,0)=(1-r)^2\xi_\alpha(r), \quad r\in(0,1),
   \end{equation}
   where
   $$
     \xi_\alpha(r)=1+(1-2\alpha)r+(1+2\alpha-4\alpha^2)r^2+(1-2\alpha)^2r^3.
   $$
   We now prove that for each fixed $\alpha\in(0,1)$, $\xi_\alpha(r)>0$, $r\in[0,1]$. For each fixed $\alpha\in(0,1)$, we have
   $$
     \xi_\alpha'(r)=(1-2\alpha)+2(1+2\alpha-4\alpha^2)r+3(1-2\alpha)^2r^2, \quad r\in[0,1].
   $$
   If $\alpha \ne 1/2$, then the above equation can be written as 
   $$
     \xi_\alpha'(r)=3(1-2\alpha)^2\left(\left(r+\frac{1+2\alpha-4\alpha^2}{3(1-2\alpha)^2}\right)^2+\frac{2(1-11\alpha+20\alpha^2-4\alpha^3-8\alpha^4)}{9(1-2\alpha)^4}\right).
   $$
   Let $\alpha_3~(=0.1137\cdots)$ be the smallest positive root of the equation $1-11\alpha+20\alpha^2-4\alpha^3-8\alpha^4=0$. If $\alpha\in(0,\alpha_3]$, then $\xi_\alpha'(r)>0$, $r\in[0,1]$. Therefore, for each fixed $\alpha\in(0,\alpha_3]$, we have
   $$
     \xi_\alpha(r)\geq\xi_\alpha(0)=1>0, \quad r\in[0,1].
   $$
   If $\alpha\in(\alpha_3,1/2)\cup(1/2,1)$, then we have
   $$
     \xi_\alpha'(r)=3(1-2\alpha)^2(r-r_1)(r-r_2),
   $$
   where
   $$
     r_1=\frac{4\alpha^2-2\alpha-1-\sqrt{2(8\alpha^4+4\alpha^3-20\alpha^2+11\alpha-1)}}{3(1-2\alpha)^2},
   $$
   and
   $$
     r_2=\frac{4\alpha^2-2\alpha-1+\sqrt{2(8\alpha^4+4\alpha^3-20\alpha^2+11\alpha-1)}}{3(1-2\alpha)^2}.
   $$
   It can be verified that $r_1<0$ for all $\alpha\in(\alpha_3,1/2)\cup(1/2,1)$. Moreover, $r_2<0$ when $\alpha\in(\alpha_3,1/2)$, whereas $r_2\in(0,1)$ for $\alpha\in(1/2,1)$. Thus, if $\alpha\in(\alpha_3,1/2)$, then $\xi_\alpha'(r)>0$, $r\in[0,1]$. Therefore, for each fixed $\alpha\in(\alpha_3,1/2)$, we have
   $$
     \xi_\alpha(r)\geq\xi_\alpha(0)=1>0, \quad r\in[0,1].
   $$
   If $\alpha\in(1/2,1)$, then
   $$
     \xi_\alpha'(r)<0 \quad {\rm{if}}~~0\leq r<r_2, \qquad \xi_\alpha'(r)>0 \quad {\rm{if}}~~r_2<r\leq 1.
   $$
   Therefore, for each fixed $\alpha\in(1/2,1)$, we have
   $$
     \xi_\alpha(r)\geq\xi_\alpha(r_2)>0.
   $$
   If $\alpha=1/2$, then $\xi_\alpha(r)=1+r^2>0$, $r\in[0,1]$. Thus, $\xi_\alpha(r)>0$, $r\in[0,1]$ for each fixed $\alpha\in(0,1)$. From \eqref{P5eq2.3} it follows that
   $$
     \min_{-\pi\leq \theta\leq\pi}{G_\alpha(r,\theta)}=G_\alpha(r,0)>0, \quad r\in(0,1),
   $$
   for each fixed $\alpha\in(0,1)$. Moreover, $G_\alpha(0,\theta)=1>0$ for all $\theta\in[-\pi,\pi]$ and $\alpha\in(0,1)$. This completes the proof of the lemma.
\end{pf}

\begin{lem}\label{P5lem2}
    For each fixed $\alpha\in(0,1)$, we have
    $$
      H_\alpha(r,\theta)>0, \quad r\in[0,R_1(\alpha)),~ \theta\in[-\pi,\pi],
    $$
    where
    \begin{align}\label{P5eq2.4}
        H_\alpha(r,\theta)& := (2\alpha-1)^2r^5+(3\alpha-2)r^4+(8\alpha^2-\alpha-3)r^3+(7\alpha-3)r^2\\& \quad +(3\alpha-2)r+1-r\big((12\alpha^2-9\alpha+1)r^3+(9\alpha-5)r^2\notag\\& \quad +(4\alpha^2+5\alpha-5)r+3\alpha+1\big)\cos{\theta}+2\alpha r^2\big((2\alpha-1)r+1\big)\cos{(2\theta)}\notag,
    \end{align}
    and $R_1(\alpha)$ is defined as in \textnormal{Theorem~\ref{P5thm1}}.
\end{lem}
\begin{pf}
   Let $\alpha\in(0,1)$. By a little computation, we have
   \begin{equation}\label{P5eq2.5}
       \frac{\partial}{\partial\theta}H_\alpha(r,\theta)=rp_\alpha(r,\theta)\sin{\theta},\quad r\in[0,1],~\theta\in[-\pi,\pi],
   \end{equation}
   where 
   \begin{align*}
       p_\alpha(r,\theta)&= 3\alpha+1+(4\alpha^2+5\alpha-5)r+(9\alpha-5)r^2+(12\alpha^2-9\alpha+1)r^3\\& \quad -8\alpha r\left(1+(2\alpha-1)r\right)\cos{\theta}.
   \end{align*}
   Since $1+(2\alpha-1)r>0$ for all $r\in[0,1]$ and $\alpha\in(0,1)$, we have
   $$
     p_\alpha(r,\theta)\geq p_\alpha(r,0)=u_\alpha(r), \quad r\in[0,1],~\theta\in[-\pi,\pi],
   $$
   where
   \begin{align*}
       u_\alpha(r)&:= 3\alpha+1+(4\alpha^2-3\alpha-5)r-(16\alpha^2-17\alpha+5)r^2\\& \quad +(12\alpha^2-9\alpha+1)r^3.
   \end{align*}
   The function $u_\alpha$ is a continuous function of $r$, $r\in[0,1]$ with
   $$
     u_\alpha(0)=3\alpha+1>0 \qquad {\rm{and}} \qquad u_\alpha(1)=-8(1-\alpha)<0.
   $$
   Therefore, by the intermediate value theorem we conclude the function $u_\alpha$ has at least one root in $(0,1)$. Let $r_0(\alpha)$ be the smallest positive root of the equation $u_\alpha(r)=0$. Then $u_\alpha(r)>0$ for $r\in[0,r_0(\alpha))$. It follows that for each fixed $\alpha\in(0,1)$, $p_\alpha(r,\theta)>0$ for all $r\in[0,r_0(\alpha))$, $\theta\in[-\pi,\pi]$. Thus, for each fixed $\alpha\in(0,1)$, we have  
   $$
    \frac{\partial H_\alpha}{\partial\theta}<0 \quad {\rm{if}}~~-\pi<\theta <0, \qquad {\rm{and}} \qquad \frac{\partial H_\alpha}{\partial\theta}>0 \quad {\rm{if}}~~0<\theta<\pi,
   $$
   for all $r\in(0,r_0(\alpha))$. Therefore,
   \begin{equation}\label{P5eq2.6}
       \min_{-\pi\leq \theta\leq\pi}{H_\alpha(r,\theta)}=H_\alpha(r,0)=(1-r)^3\left(1-(2\alpha-1)^2r^2\right)>0, \quad r\in(0,r_0(\alpha))
   \end{equation}
   for each fixed $\alpha\in(0,1)$. For $r\in[r_0(\alpha),1]$, from \eqref{P5eq2.5} we have
   $$
     \frac{\partial}{\partial \theta}H_\alpha(r,\theta)=0,
   $$
   if  
   $$
     \theta=\theta_1,~\pm\pi,~0,
   $$
   where
   \begin{equation}\label{P5eq2.7}
       \cos{\theta_1}=\frac{3\alpha+1+(4\alpha^2+5\alpha-5)r+(9\alpha-5)r^2+(12\alpha^2-9\alpha+1)r^3}{8\alpha r\left(1+(2\alpha-1)r\right)}.
   \end{equation}
   We now check the existence of such $\theta_1$ mentioned in \eqref{P5eq2.7}. For $r\in[r_0(\alpha),1]$, it can be verified that the right-hand side of \eqref{P5eq2.7} belongs to $[-1,1]$ if $v_\alpha(r)\geq 0$, where
   $$
     v_\alpha(r):=3\alpha+1+(4\alpha^2+10\alpha-6)r+(12\alpha^2-9\alpha+1)r^2.
   $$
   Now we consider four cases to check the positivity of $v_\alpha(r)$, $r\in[0,1]$ for each fixed $\alpha\in(0,1)$.

   \medskip
   \noindent Case (i). Let $\alpha\in\left(0,\left(\sqrt{17}-1\right)/8\right)$. The function $v_\alpha$ is a continuous function of $r$, $r\in[0,1]$ with
   $$
     v_\alpha(0)=3\alpha+1>0, \qquad {\rm{and}} \qquad v_\alpha(1)=4(4\alpha^2+\alpha-1)<0.
   $$
   Therefore, by the intermediate value theorem we conclude the function $v_\alpha$ has at least one root in $(0,1)$. Let $r_1(\alpha)$ be the smallest positive root of the equation $v_\alpha(r)=0$, where $\alpha\in\left(0,\left(\sqrt{17}-1\right)/8\right)$. Then $v_\alpha(r)\geq 0$ for $r\in[0,r_1(\alpha)]$, where
   $$
     r_1(\alpha)=\frac{3-5\alpha-2\alpha^2-2\sqrt{\alpha^4-4\alpha^3+7\alpha^2-6\alpha+2}}{12\alpha^2-9\alpha+1}.
   $$

   \medskip
   \noindent Case (ii). Let $\alpha\in\left[\left(\sqrt{17}-1\right)/8,1/2\right)$. For these fixed values of $\alpha$, we have $12\alpha^2-9\alpha+1<0$ and $2(3+\alpha)(2\alpha-1)<0$. For each fixed $\alpha\in(0,1)$, we have
   $$
     v_\alpha'(r)=2(3+\alpha)(2\alpha-1)+2(12\alpha^2-9\alpha+1)r, \quad r\in[0,1].
   $$
   Therefore, if $\alpha\in\left[\left(\sqrt{17}-1\right)/8,1/2\right)$, then $v_\alpha'(r)<0$ for all $r\in[0,1]$. This shows $v_\alpha$ is a decreasing function of $r$, $r\in[0,1]$. Thus,
   $$
     v_\alpha(r)\geq v_\alpha(1)=4(4\alpha^2+\alpha-1)\geq 0, \quad r\in[0,1].
   $$

   \medskip
   \noindent Case (iii). Let $\alpha\in\left[1/2,2/\left(9-\sqrt{33}\right)\right)$. For these fixed values of $\alpha$, we have
   $$
     v_\alpha'(r)=2(12\alpha^2-9\alpha+1)(r-r_1(\alpha)),\quad r\in[0,1],
   $$
   where
   $$
     r_2(\alpha)=\frac{(3+\alpha)(1-2\alpha)}{12\alpha^2-9\alpha+1}.
   $$
   Case (iii)(a). If $\alpha\in\left[1/2,1/\left(\sqrt{8}-1\right)\right]$, then $12\alpha^2-9\alpha+1<0$ and $r_2(\alpha)\in[0,1]$. This implies $r=r_2(\alpha)$ is the only critical point in $[0,1]$. Computing the values of $v_\alpha(r)$ for $r=0,1,r_2(\alpha)$, we get
   \begin{align*}
       v_\alpha(0)&= 3\alpha+1>0,\\
       v_\alpha(1)&= 4(4\alpha^2+\alpha-1)>0,\\
       v_\alpha(r_2(\alpha))&=\frac{4(1-\alpha)^2(-2+2\alpha-\alpha^2)}{12\alpha^2-9\alpha+1}>0.
   \end{align*}
   Therefore, $v_\alpha(r)>0$, $r\in[0,1]$.

   \medskip
   \noindent Case (iii)(b). If $\alpha\in\left(1/\left(\sqrt{8}-1\right),2/\left(9-\sqrt{33}\right)\right)$, then $12\alpha^2-9\alpha+1<0$ and $r_2(\alpha)>1$. This implies $v_\alpha'(r)>0$ for all $r\in[0,1]$. This shows $v_\alpha$ is an increasing function of $r$, $r\in[0,1]$. Therefore,
   $$
     v_\alpha(r)\geq v_\alpha(0)=3\alpha+1>0, \quad r\in[0,1].
   $$

   \medskip
   \noindent Case (iv). Let $\alpha\in\left[2/\left(9-\sqrt{33}\right),1\right)$. For these fixed values of $\alpha$, we have $12\alpha^2-9\alpha+1>0$ and $2(3+\alpha)(2\alpha-1)>0$. Therefore, $v_\alpha'(r)>0$ for all $r\in[0,1]$. This shows $v_\alpha$ is an increasing function of $r$, $r\in[0,1]$. Thus,
   $$
     v_\alpha(r)\geq v_\alpha(0)=3\alpha+1>0, \quad r\in[0,1].
   $$

   \medskip
   \noindent Thus, combining all the cases we get $v_\alpha(r)\geq 0$, $r\in[0,r_3(\alpha)]$ for each fixed $\alpha\in(0,1)$, where
   \begin{equation}\label{P5eq2.8}
       r_3(\alpha)=
        \begin{cases}
           \frac{3-5\alpha-2\alpha^2-2\sqrt{\alpha^4-4\alpha^3+7\alpha^2-6\alpha+2}}{12\alpha^2-9\alpha+1}, & \text{if}~\alpha\in\left(0,\left(\sqrt{17}-1\right)/8\right) \\
           1, & \text{if}~\alpha\in\left[\left(\sqrt{17}-1\right)/8,1\right).
        \end{cases}
   \end{equation}
   Therefore, if $r\in[r_0(\alpha),r_3(\alpha)]$ then there exists $\theta_1\in[-\pi,\pi]$ such that \eqref{P5eq2.7} holds. For $r\in[r_0(\alpha),r_3(\alpha)]$, computing the values of $H_\alpha(r,\theta)$ for $\theta=\pm\pi,~0,~\theta_1$, we get
   \begin{align}
       H_\alpha(r,\pm\pi)&=(1+r)^3\left(1+(6\alpha-4)r+(1-2\alpha)^2r^2\right),\notag\\
       H_\alpha(r,0)&=(1-r)^3\left(1-(1-2\alpha)^2r^2\right)>0, \notag\\
       H_\alpha(r,\theta_1)&=\frac{(1-\alpha)(1+r)^2}{16\alpha\left(1+(2\alpha-1)r\right)}\zeta_\alpha(r),\label{P5eq2.9}
   \end{align}
    where $\zeta_\alpha(r)$ is defined in the statement of Theorem~\ref{P5thm1}. The right hand side of \eqref{P5eq2.9} is strictly positive if $\zeta_\alpha(r)>0$. It can be verified that $\zeta_\alpha(r_0(\alpha))>0$ and $\zeta_\alpha(r_3(\alpha))<0$ for all $\alpha\in(0,1)$. Therefore, by the intermediate value theorem we conclude the function $\zeta_\alpha$ has at least one root in $\left(r_0(\alpha),r_3(\alpha)\right)$ for each fixed $\alpha\in(0,1)$. Let $R_1(\alpha)$ be the least value of $r\in \left(r_0(\alpha),r_3(\alpha)\right)$ that satisfies the equation $\zeta_\alpha(r)=0$. Then $\zeta_\alpha(r)>0$, $r\in[r_0(\alpha),R_1(\alpha))$. We now find the range of $R_1(\alpha)$, $\alpha\in(0,1)$. For all $r\in[0,1]$, we have
   $$
     \frac{\partial}{\partial\alpha}u_\alpha(r)=3-3r+17r^2-9r^3+8\alpha r(1-r)(1-3r), \quad \alpha\in[0,1].
   $$
   It is easy to see that $3-3r+17r^2-9r^3>0$, $r\in[0,1]$. If $r\in[0,1/3]$, then $8r(1-r)(1-3r)>0$. Therefore,
   \begin{equation}\label{P5eq2.10}
       \frac{\partial}{\partial\alpha}u_\alpha(r)>0, \quad \alpha\in[0,1],
   \end{equation}
   for all $r\in[0,1/3]$. If $r\in(1/3,1]$, then $8r(1-r)(1-3r)<0$. Therefore,
   \begin{align*}
       \frac{\partial}{\partial\alpha}u_\alpha(r)& \geq 3-3r+17r^2-9r^3+8r(1-r)(1-3r)\\& =3+5r(1-3r+3r^2),
   \end{align*}
   for all $\alpha\in[0,1]$, $r\in(1/3,1]$. Since $1-3r+3r^2>0$, $r\in(1/3,1]$, we get
   \begin{equation}\label{P5eq2.11}
       \frac{\partial}{\partial\alpha}u_\alpha(r)>0, \quad \alpha\in[0,1],
   \end{equation}
   for all $r\in(1/3,1]$. Thus, from \eqref{P5eq2.10} and \eqref{P5eq2.11} it follows that for all $r\in[0,1]$, $u_\alpha(r)$ is an increasing function of $\alpha$, $\alpha\in[0,1]$. This proves that the value of $r_0(\alpha)$ increases as the value of $\alpha$ increases, because $r_0(\alpha)$ is the smallest positive root of the equation $u_\alpha(r)=0$. Therefore,
   $$
     r_0(\alpha)>r_0(0)=3-2\sqrt{2}, \quad \alpha\in(0,1).
   $$
   Since $R_1(\alpha)\in(r_0(\alpha),r_3(\alpha))$, from \eqref{P5eq2.8} and the above inequality, it follows that $R_1(\alpha)\in\left(3-2\sqrt{2},1\right)$. Since $\zeta_\alpha(r)>0$, $r\in[r_0(\alpha),R_1(\alpha))$, from \eqref{P5eq2.9} we get $H_\alpha(r,\theta_1)>0$, $r\in[r_0(\alpha),R_1(\alpha))$. Moreover, it can be verified that $H_\alpha(r,\pm\pi)>0$, $r\in[r_0(\alpha),R_1(\alpha))$. Therefore, we have
   $$
     \min_{-\pi\leq \theta\leq\pi}{H_\alpha(r,\theta)}>0, \quad r\in \left[r_0(\alpha),R_1(\alpha)\right)
   $$
   for each fixed $\alpha\in(0,1)$. Since $H_\alpha(0,\theta)=1$ for all $\theta\in[-\pi,\pi]$ and $\alpha\in(0,1)$, by \eqref{P5eq2.6} and the above inequality, the proof of the lemma is complete.
\end{pf}

\begin{lem}\label{P5lem3}
    For each fixed $\alpha\in(0,1)$, we have
    $$
      T_\alpha(r,\theta)>0, \quad r\in[0,R_1(\alpha)),~ \theta\in[-\pi,\pi],
    $$
    where
    \begin{align}\label{P5eq2.12}
        T_\alpha(r,\theta) & := (4\alpha^3-6\alpha^2+4\alpha-1)r^5+(2\alpha^2-2\alpha+1)r^4+2(2\alpha^3+\alpha-1)r^3\\& \quad +2(2\alpha^2-\alpha+1)r^2+(2\alpha-1)r+1-\big((8\alpha^3-6\alpha^2+3\alpha-1)r^4\notag\\& \quad +(6\alpha^2-5\alpha+3)r^3+(4\alpha^2+3\alpha-3)r^2+(3\alpha+1)r\big)\cos{\theta}\notag\\& \quad +2\alpha r^2\big((2\alpha-1)r+1\big)\cos{(2\theta)},\notag
    \end{align}
    and $R_1(\alpha)$ is defined as in \textnormal{Theorem~\ref{P5thm1}}.
\end{lem}
\begin{pf}
   Let $\alpha\in(0,1)$. A straightforward calculation shows that
   \begin{equation}\label{P5eq2.13}
       \frac{\partial}{\partial\theta}T_\alpha(r,\theta)=rs_\alpha(r,\theta)\sin{\theta}, \quad r\in[0,1],~\theta\in[-\pi,\pi],
   \end{equation}
   where
   \begin{align*}
       s_\alpha(r,\theta)& =(8\alpha^3-6\alpha^2+3\alpha-1)r^3+(6\alpha^2-5\alpha+3)r^2+(4\alpha^2+3\alpha-3)r+3\alpha+1\\& \quad -8\alpha r\left((2\alpha-1)r+1\right)\cos{\theta}.
   \end{align*}
   Since $(2\alpha-1)r+1>0$ for all $r\in[0,1]$ and $\alpha\in(0,1)$, we have
   \begin{equation}\label{P5eq2.14}
       s_\alpha(r,\theta)\geq s_\alpha(r,0), \quad r\in[0,1],~\theta\in[-\pi,\pi],
   \end{equation}
   where
   \begin{equation}\label{P5eq2.15}
       s_\alpha(r,0)=3\alpha+1+(4\alpha^2-5\alpha-3)r-(10\alpha^2-3\alpha-3)r^2+(8\alpha^3-6\alpha^2+3\alpha-1)r^3.
   \end{equation}
   For each fixed $\alpha\in(0,1)$, we have
   $$
     s_\alpha'(r,0)=(4\alpha^2-5\alpha-3)-2(10\alpha^2-3\alpha-3)r+3(8\alpha^3-6\alpha^2+3\alpha-1)r^2, \quad r\in[0,1].
   $$
   If $\alpha\ne1/2$ then the above equation can be written as
   $$
     s_\alpha'(r,0)=3(2\alpha-1)(4\alpha^2-\alpha+1)(r-r_1)(r-r_2),
   $$
   where
   $$
     r_1=\frac{10\alpha^2-3\alpha-3-\sqrt{2\alpha(15-24\alpha-57\alpha^2+146\alpha^3-48\alpha^4)}}{3(2\alpha-1)(4\alpha^2-\alpha+1)},
   $$
   and
   $$
     r_2=\frac{10\alpha^2-3\alpha-3+\sqrt{2\alpha(15-24\alpha-57\alpha^2+146\alpha^3-48\alpha^4)}}{3(2\alpha-1)(4\alpha^2-\alpha+1)}.
   $$
   It can be verified that $r_1>1$ if $\alpha\in(0,1/2)$, and $r_1<0$ if $\alpha\in(1/2,1)$. Moreover, $r_2\in(0,1)$ when $\alpha\in(0,5/12)$, whereas $r_2 \geq 1$ for $\alpha\in[5/12,1/2)\cup(1/2,1)$. Since $4\alpha^2-\alpha+1>0$ for $\alpha\in(0,1)$, we have 
   $$
     s_\alpha'(r,0)<0 \quad {\rm{if}}~~0\leq r<r_2, \qquad {\rm{and}} \qquad s_\alpha'(r,0)>0 \quad {\rm{if}}~~r_2<r\leq 1,
   $$
   for each fixed $\alpha\in(0,5/12)$. Therefore, for each fixed $\alpha\in (0,5/12)$, we have
   $$
     \min_{r\in[0,1]}{s_\alpha(r,0)}=s_\alpha(r_2,0)>0.
   $$
   If $\alpha\in [5/12,1/2)\cup(1/2,1)$, then 
   $$
     s_\alpha'(r,0)\leq 0, \quad r\in[0,1].
   $$
   This shows $s_\alpha(r,0)$ is a decreasing function of $r$, $r\in[0,1]$ for each fixed $\alpha\in [5/12,1/2)\cup(1/2,1)$. If $\alpha\in[5/12,1/2)$, then
   $$
     {s_\alpha(r,0)}\geq s_\alpha(1,0)=4\alpha(1-\alpha)(1-2\alpha)>0, \quad r\in[0,1].
   $$
   If $\alpha\in(1/2,1)$ then
   $$
     s_\alpha(0,0)=3\alpha+1>0 \qquad {\rm{and}} \qquad s_\alpha(1,0)=4\alpha(1-\alpha)(1-2\alpha)<0.
   $$
   Therefore, by the intermediate value theorem it follows that $s_\alpha(r,0)$ has at least one root in $(0,1)$ for each fixed $\alpha\in(1/2,1)$. Let $r_3$ be the smallest positive root of the equation $s_\alpha(r,0)=0$, $\alpha\in(1/2,1)$. Then $s_\alpha(r,0)>0$ for $r\in[0,r_3)$, $\alpha\in(1/2,1)$. Thus, from \eqref{P5eq2.14} it follows that
   $$
     s_\alpha(r,\theta)>0, \quad r\in[0,r_4), ~\theta\in[-\pi,\pi]
   $$
   for each fixed $\alpha\in(0,1/2)\cup(1/2,1)$, where
   \begin{equation}\label{P5eq2.16}
        r_4=
     \begin{cases}
        1, & \text{if} ~~\alpha\in(0,1/2)\\
        r_3, & \text{if} ~~\alpha\in(1/2,1),   
     \end{cases}
   \end{equation}
   with $r_3$ being the smallest positive root of the equation $s_\alpha(r,0)=0$. Here $s_\alpha(r,0)$ is defined in \eqref{P5eq2.15}. Therefore, for each fixed $\alpha\in(0,1/2)\cup(1/2,1)$, from \eqref{P5eq2.13} we have
   $$
     \frac{\partial T_\alpha}{\partial\theta}<0 \quad {\rm{if}}~~-\pi<\theta <0, \qquad {\rm{and}} \qquad \frac{\partial T_\alpha}{\partial\theta}>0 \quad {\rm{if}}~~0<\theta<\pi,
   $$
   for all $r\in(0,r_4)$. It follows that for each fixed $\alpha\in(0,1/2)\cup(1/2,1)$,
   $$
     \min_{-\pi\leq \theta\leq\pi}{T_\alpha(r,\theta)}=T_\alpha(r,0)=(1-r)w_\alpha(r), \quad r\in(0,r_4),
   $$
   where
   \begin{align*}
       w_\alpha(r)& := 1-(1+\alpha)r+4(1-\alpha)r^2+(4\alpha^3-2\alpha^2+\alpha-1)r^3\\& \quad +(1-4\alpha+6\alpha^2-4\alpha^3)r^4.
   \end{align*}
   For each fixed $\alpha\in(0,1)$, we have
   \begin{align*}
       w_\alpha'(r)& = -1-\alpha+8(1-\alpha)r+3(4\alpha^3-2\alpha^2+\alpha-1)r^2\\& \quad +4(1-2\alpha)(1-2\alpha+2\alpha^2)r^3, \quad r\in[0,1].
   \end{align*}
   It can be verified that there exists a unique real number $r_5\in[0,1]$ such that $w_\alpha'(r_5)=0$ and $w_\alpha(r_5)>0$ for each fixed $\alpha\in(0,1)$. Moreover, $w_\alpha(0)=1>0$ and $w_\alpha(1)=4(1-\alpha)^2>0$ for $\alpha\in(0,1)$. This implies
   $$
     w_\alpha(r)>0, \quad r\in[0,1], ~\alpha\in(0,1).
   $$
   Thus,
   \begin{equation}\label{P5eq2.17}
       \min_{-\pi\leq \theta\leq\pi}{T_\alpha(r,\theta)}=T_\alpha(r,0)=(1-r)w_\alpha(r)>0, \quad r\in(0,r_4)
   \end{equation}
   for each fixed $\alpha\in(0,1/2)\cup(1/2,1)$. If $\alpha=1/2$, then from \eqref{P5eq2.13} we have
   $$
     \frac{\partial}{\partial\theta}T_\alpha(r,\theta)=\frac{r}{2}(4r^2-r+5-8r\cos{\theta})\sin{\theta}.
   $$
   It can be readily seen that $4r^2-r+5-8r\cos{\theta}\geq 4r^2-9r+5>0$ for $r\in[0,1)$, $\theta\in[-\pi,\pi]$. Thus, for $\alpha=1/2$,
   $$
    \frac{\partial T_\alpha}{\partial\theta}<0 \quad {\rm{if}}~~-\pi<\theta <0, \qquad {\rm{and}} \qquad \frac{\partial T_\alpha}{\partial\theta}>0 \quad {\rm{if}}~~0<\theta<\pi,
   $$
   for all $r\in(0,1)$. Therefore,
   \begin{equation}\label{P5eq2.18}
       \min_{-\pi\leq \theta\leq\pi}{T_\alpha(r,\theta)}=T_\alpha(r,0)=\frac{1}{2}(1-r)(2-3r+4r^2-r^3)>0, \quad r\in(0,1),
   \end{equation}
   for $\alpha=1/2$. Thus, by \eqref{P5eq2.17} and \eqref{P5eq2.18} we get
   \begin{equation}\label{P5eq2.19}
       \min_{-\pi\leq \theta\leq\pi}{T_\alpha(r,\theta)}>0, \quad r\in(0,r_4)
   \end{equation}
   for each fixed $\alpha\in(0,1)$, where $r_4$ is defined in \eqref{P5eq2.16}. It can be verified that for each fixed $\alpha\in(0,1)$,
    $$
        \zeta_\alpha'(r)<0, \quad r\in\left(3-2\sqrt{2},1\right],
    $$
    where $\zeta_\alpha$ is defined in the statement of Theorem~\ref{P5thm1}. This shows $\zeta_\alpha$ is a strictly decreasing function of $r$, $r\in\left(3-2\sqrt{2},1\right]$. By a little computation, we see that
    $$
      \zeta_\alpha(r_4)<0=\zeta_\alpha(R_1(\alpha)), \quad \alpha\in(0,1).
    $$
    Since $r_4\in\left(3-2\sqrt{2},1\right]$ and $R_1(\alpha)\in\left(3-2\sqrt{2},1\right]$, $\alpha\in(0,1)$, and $\zeta_\alpha$ is a strictly decreasing function of $r$, $r\in\left(3-2\sqrt{2},1\right]$, from the above inequality it follows that
    $$
      r_4>R_1(\alpha),\quad \alpha\in(0,1).
    $$
    Thus, by \eqref{P5eq2.19} and the above inequality we get
    $$
      T_\alpha(r,\theta)>0, \quad r\in(0,R_1(\alpha)), ~ \theta\in[-\pi,\pi]
    $$
   for each fixed $\alpha\in(0,1)$. Moreover, $T_\alpha(0,\theta)=1>0$ for all $\theta\in[-\pi,\pi]$, $\alpha\in(0,1)$. This completes the proof of the lemma.
\end{pf}

Using Lemmas~\ref{P5lem1},~\ref{P5lem2},~\ref{P5lem3} and the duality between the sets $\mathcal{M}$ and $\mathcal{S}t$, we prove the next lemma.  

\begin{lem}\label{P5lem4}
   Let $\alpha \in (0,1)$. If $\phi\in \mathcal{K}$ and $g\in \mathcal{S}t(\alpha)$, then for each complex number $\sigma \in \partial \D$ and $\beta \in \partial \D$, we have
    $$
     \phi(z) \ast \left(\frac{1+\left((1-\alpha)\beta-\alpha\right)\sigma R_1(\alpha)z}{1-\sigma R_1(\alpha)z} g\left(R_1(\alpha)z\right)\right) \ne 0, \quad 0<|z|<1,
    $$
    where $R_1(\alpha)$ is defined as in \textnormal{Theorem~\ref{P5thm1}}.
\end{lem}
\begin{pf}
   Let $\phi\in \mathcal{K}$ and $g\in\mathcal{S}t(\alpha)$, $\alpha \in (0,1)$. For each complex number $\sigma \in \partial \D$ and $\beta \in \partial \D$, let
   $$
     h(z):=\frac{1+\left((1-\alpha)\beta-\alpha\right)\sigma z}{1-\sigma z} g(z), \quad z\in\D.
   $$
   Then we need to prove 
   $$
    \phi(z)\ast h(R_1(\alpha)z) \ne 0, \quad 0<|z|<1,
   $$
   where $R_1(\alpha)$ is defined in the statement of the lemma. A straightforward calculation shows that
   \begin{equation}\label{P5eq2.20}
       \frac{zh'(z)}{h(z)}=\frac{\left((1-\alpha)\beta-\alpha\right)\sigma z}{1+\left((1-\alpha)\beta-\alpha\right)\sigma z}+\frac{\sigma z}{1-\sigma z}+\frac{zg'(z)}{g(z)}, \quad z\in \D.
   \end{equation}
   Let
   $$
     \psi(z):=z\left(\frac{g(z)}{z}\right)^{\frac{1}{1-\alpha}}, \quad z\in\D.
   $$
   By a little computation, we see that
   \begin{equation}\label{P5eq2.21}
        \frac{z\psi'(z)}{\psi(z)}=\frac{1}{1-\alpha}\left(\frac{zg'(z)}{g(z)}\right)-\frac{\alpha}{1-\alpha},\quad z\in\D.
   \end{equation}
   Since $g\in\mathcal{S}t(\alpha)$, we have ${\rm{Re}}\,\left(zg'(z)/g(z)\right)>\alpha$, $z\in\D$. Thus, from \eqref{P5eq2.21} we get
   $$
     {\rm{Re}}\,\left(\frac{z\psi'(z)}{\psi(z)}\right)>\frac{\alpha}{1-\alpha}-\frac{\alpha}{1-\alpha}=0, \quad z\in\D.
   $$
   This shows $\psi\in\mathcal{S}t$. It follows that $z\psi'(z)/\psi(z)\in \mathcal{P}$ - the class of analytic functions $p$ in $\D$ such that $p(0)=1$ and ${\rm{Re}}\,p(z)>0$, $z\in\D$. Therefore,
   $$
    {\rm{Re}}\,\left(\frac{z\psi'(z)}{\psi(z)}\right)\geq \frac{1-|z|}{1+|z|},\quad z\in\D;
   $$
   (see \cite[Theorem~2.1.3]{graham}). Thus, with the help of the above inequality, from \eqref{P5eq2.21} we get
   $$
    {\rm{Re}}\,\left(\frac{zg'(z)}{g(z)}\right)\geq \frac{1+(2\alpha-1)|z|}{1+|z|},\quad z\in\D.
   $$
   Applying the above inequality, from \eqref{P5eq2.20} it follows that
   \begin{equation*}
       {\rm{Re}}\,\left(\frac{zh'(z)}{h(z)}\right)\geq {\rm{Re}}\,\left(\frac{\left((1-\alpha)\beta-\alpha\right)\sigma z}{1+\left((1-\alpha)\beta-\alpha\right)\sigma z}+\frac{\sigma z}{1-\sigma z}\right)+\frac{1+(2\alpha-1)|z|}{1+|z|},\quad z\in\D.
   \end{equation*}
   Let $\beta=e^{it}$ and $\sigma z=re^{i\theta}$, where $r\in(0,1)$, $\theta\in[-\pi,\pi)$, and $t\in[-\pi,\pi)$. Then after elementary computations, we obtain
   \begin{equation}\label{P5eq2.22}
       {\rm{Re}}\,\left(\frac{zh'(z)}{h(z)}\right)\geq \frac{Q_\alpha(r,\theta,t)}{(1+|z|)\left|1+\left((1-\alpha)\beta-\alpha\right)\sigma z\right|^2\left|1-\sigma z\right|^2},\quad z\in\D,
   \end{equation}
   where 
   $$
     Q_\alpha(r,\theta,t)=T_\alpha(r,\theta)+A_\alpha(r,\theta)\sin{t}+B_\alpha(r,\theta)\cos{t},
   $$
   with $T_\alpha$ as defined in \eqref{P5eq2.12} and
   \begin{align*}
       A_\alpha(r,\theta) &= (1-\alpha)\big((3-4\alpha)r^4-r^3+(1-4\alpha)r^2-3r\big)\sin{\theta}\\& \quad +2(1-\alpha)r^2\big((2\alpha-1)r+1\big)\sin{(2\theta)},\\
       B_\alpha(r,\theta) &= -2(1-\alpha)r^2\left(2\alpha+1+(2\alpha^2+2\alpha-1)r+\alpha r^2+(2\alpha^2-\alpha)r^3\right)\\& \quad +(1-\alpha)r\left(3+(4\alpha-1)r+(6\alpha+1)r^2+(8\alpha^2+2\alpha-3)r^3\right)\cos{\theta}\\& \quad -2(1-\alpha)r^2\big((2\alpha-1)r+1\big)\cos{(2\theta)}.
   \end{align*}
   If we let $A_\alpha(r,\theta)=M\cos{t_1}$ and $B_\alpha(r,\theta)=M\sin{t_1}$, then we have
   $$
    Q_\alpha(r,\theta,t)=T_\alpha(r,\theta)+M\sin{(t+t_1)}.
   $$
   Thus,
   \begin{equation}\label{P5eq2.23}
       Q_\alpha(r,\theta,t)\geq T_\alpha(r,\theta)-M=T_\alpha(r,\theta)-\sqrt{(A_\alpha(r,\theta))^2+(B_\alpha(r,\theta))^2}.
   \end{equation}
   A straightforward calculation yields
   \begin{equation}\label{P5eq2.24}
       \left(T_\alpha(r,\theta)\right)^2-\left((A_\alpha(r,\theta))^2+(B_\alpha(r,\theta))^2\right)=G_\alpha(r,\theta)H_\alpha(r,\theta),
   \end{equation}
   where $G_\alpha$ and $H_\alpha$ are defined in \eqref{P5eq2.1} and \eqref{P5eq2.4}, respectively. By Lemmas~\ref{P5lem1} and \ref{P5lem2}, it follows that
   $$
     G_\alpha(r,\theta)H_\alpha(r,\theta)>0, \quad r\in \left[0,R_1(\alpha)\right),~\theta\in[-\pi,\pi)
   $$
   for each fixed $\alpha\in(0,1)$. Applying the above inequality, from \eqref{P5eq2.24} we get
   \begin{equation}\label{P5eq2.25}
       \left(T_\alpha(r,\theta)\right)^2-\left((A_\alpha(r,\theta))^2+(B_\alpha(r,\theta))^2\right)>0, \quad r\in \left[0,R_1(\alpha)\right),~\theta\in[-\pi,\pi)
   \end{equation}
   for each fixed $\alpha\in(0,1)$. On the other hand, by Lemma~\ref{P5lem3}, we have
   $$
     T_\alpha(r,\theta)+\sqrt{(A_\alpha(r,\theta))^2+(B_\alpha(r,\theta))^2}>0, \quad r\in[0,R_1(\alpha)),~~\theta\in[-\pi,\pi)
   $$
   for each fixed $\alpha\in(0,1)$. Therefore, from \eqref{P5eq2.25} and the above inequality it follows that
   $$
    T_\alpha(r,\theta)-\sqrt{(A_\alpha(r,\theta))^2+(B_\alpha(r,\theta))^2}>0, \quad r\in[0,R_1(\alpha)),~~\theta\in[-\pi,\pi)
   $$
   for each fixed $\alpha\in(0,1)$. Using this, from \eqref{P5eq2.23} we get
   $$
    Q_\alpha(r,\theta,t)>0, \quad r\in[0,R_1(\alpha)),~~\theta\in[-\pi,\pi),~~t\in[-\pi,\pi)
   $$
   for each fixed $\alpha\in(0,1)$. Thus, from \eqref{P5eq2.22} we have
   $$
    {\rm{Re}}\,\left(\frac{zh'(z)}{h(z)}\right)>0,\quad |z|=r<R_1(\alpha),~\alpha\in(0,1).
   $$
   This implies
   $$
    \frac{1}{R_1(\alpha)}h(R_1(\alpha)z) \in \mathcal{S}t, \quad \alpha\in(0,1).
   $$
   It was shown in Section~\ref{P5sec1} that if $f\in \mathcal{M}$, then $(f\ast g)(z)\ne 0$ for all $g\in \mathcal{S}t$, and the set $\mathcal{M}$ contains the set $\mathcal{K}$. Since $\phi \in \mathcal{K}$, it follows that
   $$
    \phi(z) \ast \left(\frac{1}{R_1(\alpha)}h(R_1(\alpha)z)\right)\ne 0, \quad 0<|z|<1, ~ \alpha\in(0,1).
   $$
   This proves the lemma. 
\end{pf}
\begin{lem}\label{P5lem5}
     Let $\phi$ and $g$ be analytic functions defined in $\D$ and satisfy $\phi(0)=0=g(0)$ with $\phi'(0)\ne 0$ and $g'(0)\ne 0$. If for each complex number $\sigma \in \partial \D$ and $\beta \in \partial \D$,
     \begin{equation}\label{P5eq2.26}
          \phi(z) \ast \left(\frac{1+\left((1-\alpha)\beta-\alpha\right)\sigma Rz}{1-\sigma Rz} g(Rz)\right) \ne 0, \quad 0<|z|<1,
     \end{equation}
    where $R\in(0,1)$ and $\alpha\in (0,1)$, then 
    $$
     {\rm{Re}}\, \left(\frac{\phi(z) \ast \left(\frac{1+(1-2\alpha)\sigma Rz}{1-\sigma Rz}g(Rz)\right)}{\phi(z)\ast g(Rz)}\right)>0, \quad z\in\D.
    $$
\end{lem}
\begin{pf}
     We first observe that if $\beta=-1$, then by \eqref{P5eq2.26} we get
    $$
     \phi(z) \ast g(Rz) \ne 0, \quad 0<|z|<1.
    $$
    For each complex number $\sigma \in \partial \D$ and $\beta \in \partial \D$, we have
    \begin{align*}
         &\phi(z) \ast \left(\frac{1+\left((1-\alpha)\beta-\alpha\right)\sigma Rz}{1-\sigma Rz} g(Rz)\right)\\=& \left(\frac{1+\beta}{2}\right)\phi(z) \ast \left(\frac{1+(1-2\alpha)\sigma Rz}{1-\sigma Rz}g(Rz)\right)
         +\left(\frac{1-\beta}{2}\right) \phi(z) \ast g(Rz), \quad z\in \D.
    \end{align*}
   Dividing the above equation by $\phi(z) \ast g(Rz)$, we get
   $$
     \left(\frac{1+\beta}{2}\right)\frac{\phi(z) \ast \left(\frac{1+(1-2\alpha)\sigma Rz}{1-\sigma Rz}g(Rz)\right)}{\phi(z)\ast g(Rz)}=\frac{\phi(z) \ast \left(\frac{1+\left((1-\alpha)\beta-\alpha\right)\sigma Rz}{1-\sigma Rz} g(Rz)\right)}{\phi(z) \ast g(Rz)}-\frac{1-\beta}{2},
   $$
   for $0<|z|<1$. If we assume $\beta \ne -1$, then by \eqref{P5eq2.26}, and from the above equation we get
    $$
     \frac{\phi(z) \ast \left(\frac{1+(1-2\alpha)\sigma Rz}{1-\sigma Rz}g(Rz)\right)}{\phi(z)\ast g(Rz)} \ne -\frac{1-\beta}{1+\beta}, \quad 0<|z|<1.
    $$
    Thus, the function in the left-hand side of the above relation does not take any value on the imaginary axis, but, clearly has the value $1$ at $z=0$. Hence, the lemma follows.
\end{pf}

\begin{pf}[Proof of Theorem~1]
    Let $g\in \mathcal{S}t(\alpha)$, $\alpha\in(0,1)$. Then from \eqref{P5eq2.21} we have
    \begin{equation}\label{P5eq2.27}
        \frac{zg'(z)}{g(z)}=\alpha+(1-\alpha)\left(\frac{z\psi'(z)}{\psi(z)}\right), \quad z\in\D,
    \end{equation}
    where $\psi\in\mathcal{S}t$. By the Herglotz formula we have
    $$
     \frac{z\psi'(z)}{\psi(z)}=\int\limits_{\partial\D}\frac{1+\sigma z}{1-\sigma z}~d\mu(\sigma), \quad z\in\D,
    $$
    where $\mu$ is a probability measure on $\partial\D$. Thus, from \eqref{P5eq2.27} we get
    $$
     \frac{zg'(z)}{g(z)}=\alpha+(1-\alpha)\int\limits_{\partial\D}\frac{1+\sigma z}{1-\sigma z}~d\mu(\sigma), \quad z\in\D,
    $$
    which simplifies to
    $$
     zg'(z)=\int\limits_{\partial\D}\frac{1+(1-2\alpha)\sigma z}{1-\sigma z}g(z)~d\mu(\sigma), \quad z\in\D.
    $$
    By a little computation, from the above equation we get
    \begin{equation}\label{P5eq2.28}
        \frac{f(z)\ast \left(Rzg'(Rz)\right)}{f(z)\ast g(Rz)}=\int\limits_{\partial\D}\frac{f(z)\ast \left(\frac{1+(1-2\alpha)\sigma Rz}{1-\sigma Rz}g(Rz)\right)}{f(z)\ast g(Rz)}~d\mu(\sigma), \quad z\in\D,
    \end{equation}
    where $f\in\mathcal{S}$ and $R\in (0,1)$. Now, it is easy to see that 
    \begin{equation*}
        \frac{Rz(f\ast g)'(Rz)}{(f\ast g)(Rz)}=\frac{f(z)\ast \left(Rzg'(Rz)\right)}{f(z)\ast g(Rz)}, \quad 0<|z|<1.
    \end{equation*}
    Thus, from \eqref{P5eq2.28} and the above equality we get
    \begin{equation}\label{P5eq2.29}
        \frac{Rz(f\ast g)'(Rz)}{(f\ast g)(Rz)}=\int\limits_{\partial\D}\frac{f(z)\ast \left(\frac{1+(1-2\alpha)\sigma Rz}{1-\sigma Rz}g(Rz)\right)}{f(z)\ast g(Rz)}~d\mu(\sigma), \quad z\in\D.
    \end{equation}
    Since $f\in \mathcal{S}$ and the radius of close-to-convexity of $\mathcal{S}$ is $r_{cc}$, we have
    $$
     \frac{1}{r_{cc}}f(r_{cc}z)\in \mathcal{K}, \quad z\in\D.
    $$
    By Lemma~\ref{P5lem4}, it follows that for each complex number $\sigma \in \partial \D$ and $\beta \in \partial \D$,
    $$
     \frac{1}{r_{cc}}f(r_{cc}z)\ast \left(\frac{1+\left((1-\alpha)\beta-\alpha\right)\sigma R_1(\alpha)z}{1-\sigma R_1(\alpha)z} g\left(R_1(\alpha)z\right)\right) \ne 0, \quad 0<|z|<1,
    $$
    which is equivalent to
    $$
      f(z)\ast \left(\frac{1+\left((1-\alpha)\beta-\alpha\right)\sigma R_1(\alpha)r_{cc}z}{1-\sigma R_1(\alpha)r_{cc}z} g\big(R_1(\alpha)r_{cc}z\big)\right) \ne 0, \quad 0<|z|<1.
    $$
    Consequently, applying Lemma~\ref{P5lem5}, from the above inequality we get
    $$
      {\rm{Re}}\, \left(\frac{f(z) \ast \left(\frac{1+(1-2\alpha)\sigma R_1(\alpha)r_{cc}z}{1-\sigma R_1(\alpha)r_{cc}z}g\big(R_1(\alpha)r_{cc}z\big)\right)}{f(z)\ast g\big(R_1(\alpha)r_{cc}z\big)}\right)>0, \quad z\in\D.
    $$
    Using this, from \eqref{P5eq2.29} we get
    $$
      {\rm{Re}}\,\left(\frac{R_1(\alpha)r_{cc}z(f\ast g)'\big(R_1(\alpha)r_{cc}z\big)}{(f\ast g)\big(R_1(\alpha)r_{cc}z\big)}\right)>0, \quad z\in\D,
    $$
    i.e.
    $$
      {\rm{Re}}\,\left(\frac{z(f\ast g)'(z)}{(f\ast g)(z)}\right)>0, \quad |z|<R_1(\alpha)r_{cc}.
    $$
    This implies $f\ast g$ is starlike in $|z|<R_1(\alpha)r_{cc}$. As we have discussed in Section~\ref{P5sec1}, a previously known lower bound for the radius of starlikeness of $\mathcal{S}\ast\mathcal{S}t(\alpha)$, $\alpha\in[0,1)$, is $R_0(\alpha)$, where $R_0(\alpha)$ is defined in \eqref{P5eq1.1}. Thus, if $f\in\mathcal{S}$ and $g\in\mathcal{S}t(\alpha)$, $\alpha\in(0,1)$, then $f\ast g$ is starlike in $|z|<R(\alpha):=\max\left\{R_0(\alpha),R_1(\alpha)r_{cc}\right\}$. It can be verified that for each fixed $\alpha\in(0,1)$,
    $$
        \zeta_\alpha'(r)<0, \quad r\in\left(3-2\sqrt{2},1\right),
    $$
    where $\zeta_\alpha$ is defined in the statement of Theorem~\ref{P5thm1}. Therefore, for each fixed $\alpha\in(0,1)$, $\zeta_\alpha$ is a strictly decreasing function of $r$, $r\in\left(3-2\sqrt{2},1\right)$. We now determine the value of $R(\alpha)$, $\alpha\in(0,1)$. Let $\alpha_0$ be the smallest positive root of the equation
    $$
      \zeta_\alpha\left(\frac{2-\sqrt{3}}{r_{cc}}\right)=0.
    $$
    Then a straightforward calculation shows that
    $$
      \zeta_\alpha\left(\frac{2-\sqrt{3}}{r_{cc}}\right)\leq0=\zeta_\alpha\left(R_1(\alpha)\right), \quad \alpha\in(0,\alpha_0].
    $$
    Since $\zeta_\alpha$ is a strictly decreasing function of $r$, $r\in\left(3-2\sqrt{2},1\right)$ and $3-2\sqrt{2}<(2-\sqrt{3})/{r_{cc}}<1$, $R_1(\alpha)\in\left(3-2\sqrt{2},1\right)$, $\alpha\in(0,1)$, from the above inequality it follows that
    $$
      \frac{2-\sqrt{3}}{r_{cc}}\geq R_1(\alpha), \quad \alpha\in(0,\alpha_0],
    $$
    i.e.
    \begin{equation}\label{P5eq2.30}
        2-\sqrt{3}\geq R_1(\alpha)r_{cc}, \quad \alpha\in(0,\alpha_0].
    \end{equation}
    Let $\alpha_1~(=0.3349\cdots)$ be the smallest positive root of the equation $20\alpha^4-52\alpha^3+15\alpha^2+12\alpha-4=0$. Let $b=\left(2+\sqrt{3}\right)\tanh{(\pi/4)}$ and $\alpha_2=\left(2-3b+\sqrt{5b^2+4b}\right)/4=0.2404\cdots$. Then a little computation yields 
    $$
      \zeta_\alpha\left(\frac{2-\sqrt{3}}{r_{cc}}\right)>0, \quad \alpha\in(\alpha_0,\alpha_2],
    $$
    $$
      \zeta_\alpha\left(\frac{\tanh{(\pi/4)}}{r_{cc}\left(2-3\alpha+\sqrt{5\alpha^2-8\alpha+3}\right)}\right)>0, \quad \alpha\in(\alpha_2,\alpha_1],
    $$
    $$
      \zeta_\alpha\left(\frac{\tanh{(\pi/4)}(5\alpha-1)^{1/2}}{r_{cc}\left(4\alpha^2-\alpha+1+4\alpha\sqrt{\alpha^2-3\alpha+2}\right)^{1/2}}\right)>0, \quad \alpha\in(\alpha_1,1).
    $$
    Thus, we have
    $$
      \zeta_\alpha\left(\frac{R_0(\alpha)}{r_{cc}}\right)>0=\zeta_\alpha\left(R_1(\alpha)\right), \quad \alpha\in(\alpha_0,1).
    $$ 
    Since $\zeta_\alpha$ is a strictly decreasing function of $r$, $r\in\left(3-2\sqrt{2},1\right)$ and $3-2\sqrt{2}<R_0(\alpha)/r_{cc}<1$, $R_1(\alpha)\in\left(3-2\sqrt{2},1\right)$, $\alpha\in(0,1)$,  from the above inequality it follows that
    $$
      \frac{R_0(\alpha)}{r_{cc}}<R_1(\alpha), \quad \alpha\in(\alpha_0,1),
    $$
    i.e.
    \begin{equation}\label{P5eq2.31}
        R_0(\alpha)<R_1(\alpha)r_{cc}, \quad \alpha\in(\alpha_0,1).
    \end{equation}
    Thus, from \eqref{P5eq2.30} and \eqref{P5eq2.31} we get the required value of $R(\alpha)$, $\alpha\in(0,1)$. This completes the proof of Theorem~\ref{P5thm1}.
\end{pf}
\section{\bf Statements and  Declarations}
\noindent{\bf Funding:}\, Souvik Biswas has received research support from IIT Kharagpur.\\
\noindent{\bf Competing interests:}\, The authors have no relevant financial or non-financial interests to disclose.\\
\noindent{\bf Data availability statements:}\, No datasets were used during the preparation of this manuscript.\\
\noindent{\bf Author Contributions:}\, Both authors have contributed equally to the article.


\begin{thebibliography}{99}

\bibitem{alex}{\sc J. W. Alexander}: Functions which map the interior of the unit circle upon simple regions, \textit{Ann. of Math. (\textsl{2})} {\bf 17} (1915), no. 1, 12--22.

\bibitem{ba}{\sc \'{A}. Baricz, D. K. Dimitrov, H. Orhan, and N. Ya\u{g}mur}: Radii of starlikeness of some special functions, \textit{Proc. Amer. Math. Soc.} {\bf 144} (2016), no. 8, 3355--3367.

\bibitem{kumar}{\sc \'{A}. Baricz, P. Kumar, and S. Singh}: On starlikeness of regular {C}oulomb wave functions, \textit{Proc. Amer. Math. Soc.} {\bf 151} (2023), no. 6, 2325--2338.

\bibitem{szak}{\sc \'{A}. Baricz, A. Szak\'{a}l, R. Sz\'{a}sz, and N. Ya\u{g}mur}: Radii of starlikeness and convexity of a product and cross-product of {B}essel functions, \textit{Results Math.} {\bf 73} (2018), no. 2, Paper No. 62, 34 pp.

\bibitem{goodman}{\sc A. W. Goodman}: The rotation theorem for starlike univalent functions, \textit{Proc. Amer. Math. Soc.} {\bf 4} (1953), 278--286.

\bibitem{graham}{\sc I. Graham and G. Kohr}: Geometric function theory in one and higher dimensions, \textit{Monographs and Textbooks in Pure and Applied Mathematics}, {\bf 255.} Marcel Dekker, Inc., New York (2003), {\rm{xviii}}+530.

\bibitem{roth}{\sc R. Greiner and O. Roth}: On the radius of convexity of linear combinations of univalent functions and their derivatives, \textit{Math. Nachr.} {\bf 254/255} (2003), 153--164.

\bibitem{grunsky}{\sc H. Grunsky}: Zwei Bemerkungen zur konformen Abbildung, \textit{Jber. Deutsch. Math.-Verein.} {\bf 43} (1934), 140--143.

\bibitem{ling}{\sc Y. Ling and S. Ding}: On radii of starlikeness and convexity for convolutions of starlike functions, \textit{Internat. J. Math. Math. Sci.} {\bf 20} (1997), no. 2, 403--404.

\bibitem{nev}{\sc R. Nevanlinna}: {\"U}ber die konforme Abbildung von Sterngebieten, \textit{\"Oversikt av Finska Vetenskaps-Soc. F\"orh.} {\bf 63(A)} (1920-21), 1--21.

\bibitem{ob}{\sc M. Obradovi\'{c} and S. Ponnusamy}: Starlikeness of sections of univalent functions, \textit{Rocky Mountain J. Math.} {\bf 44} (2014), no. 3, 1003--1014.

\bibitem{rob}{\sc M. S. Robertson}: On the theory of univalent functions, \textit{Ann. Math.} {\bf 37} (1936), 374--408.

\bibitem{ru}{\sc S. Ruscheweyh}: Duality for {H}adamard products with applications to extremal problems for functions regular in the unit disc, \textit{Trans. Amer. Math. Soc.} {\bf 210} (1975), 63--74.

\bibitem{sheil}{\sc S. Ruscheweyh and T. Sheil-Small}: Hadamard Products of Schlicht Functions and the P\'{o}lya-Schoenberg conjecture, \textit{Comment. Math. Helv.} {\bf 48} (1973), 119--135.

\bibitem{schild}{\sc A. Schild}: On starlike functions of order $\alpha$, \textit{Am. J. Math.} {\bf 87} (1965), no. 1, 65--70.

\bibitem{small}{\sc T. Sheil-Small, H. Silverman and E. Silvia}: Convolution multiplier and starlike functions, \textit{J. Analyse Math.} {\bf 41} (1982), 181--192.

\bibitem{goel}{\sc V. Singh and R. M. Goel}: On radii of convexity and starlikeness of some classes of functions, \textit{J. Math. Soc. Japan} {\bf 23} (1971), 323--339.

\bibitem{sokol}{\sc J. Sok\'{o}\l}: Radius problems in the class {${\mathscr{SL}}^*$}, \textit{Appl. Math. Comput.} {\bf 214} (2009), no. 2, 569--573.

\bibitem{str}{\sc E. Strohh{\"a}cker}: Beitr{\"a}ge zur Theorie der schlichten Funktionen, \textit{Math. Z.} {\bf 37} (1933), no. 1, 356--380.

\end{thebibliography}
\end{document}